\documentclass[11pt,reqno,a4paper]{amsart}
\usepackage{euscript}
\usepackage{color}
\usepackage[bookmarksnumbered=true]{hyperref} 
\hypersetup{
colorlinks = true,
linkcolor = blue,
anchorcolor = blue,
citecolor = blue,
filecolor = blue,
urlcolor = blue
}

\usepackage{graphicx,caption}

\newtheorem{theorem}{Theorem}
\newtheorem{proposition}[theorem]{Proposition}
\newtheorem{corollary}[theorem]{Corollary}
\newtheorem{lemma}{Lemma}
\theoremstyle{definition}

\newtheorem{example}{Example}

\renewcommand{\epsilon}{\varepsilon}
\renewcommand{\phi}{\varphi}
\def\R{\mathbb{R}}
\def\N{\mathbb{N}}
\def\Z{\mathbb{Z}}
\DeclareMathOperator{\Span}{span}
\DeclareMathOperator{\interior}{int}
\DeclareMathOperator{\diam}{diam}
\def\cH{\EuScript{H}}
\def\cM{\EuScript{M}}
\def\cMe{\EuScript{M}_{\rm erg}}
\def\cU{\EuScript{U}}
\def\cW{\EuScript{W}}
\def\cX{\EuScript{X}}

\begin{document}
	
\title{Higher-dimensional Nonlinear Thermodynamic Formalism}
	
\begin{abstract}
We introduce a higher-dimensional version of the nonlinear thermodynamic formalism introduced by Buzzi and Leplaideur, in which a potential is replaced by a family of potentials. In particular, we establish a corresponding variational principle and we discuss the existence, characterization, and number of equilibrium measures for this higher-dimensional version.
\end{abstract}

\begin{thanks}
{Partially supported by FCT/Portugal through
UID/MAT/04459/2019. C.H. was supported by FCT/Portugal through the grant PD/BD/135523/2018.}
\end{thanks}

\author{Luis Barreira}\author{Carllos Holanda}
\address{Departamento de Matem\'atica, Instituto Superior T\'ecnico, Universidade de Lisboa, 1049-001 Lisboa, Portugal}
\email{barreira@math.tecnico.ulisboa.pt}
\email{c.eduarddo@gmail.com}

\maketitle

%----------------------
\section{Introduction}
%----------------------

Recently, Buzzi and Leplaideur \cite{BL20} introduced a variation of the thermodynamic formalism, which they called \emph{nonlinear thermodynamic formalism}. Roughly speaking, this amounts to compute the topological pressure replacing Birkhoff sums by images of them under a given function (that may be nonlinear and thus the name). Our main aim is twofold:
\begin{enumerate}
\item
to introduce a higher-dimensional version of their notion of topological pressure, replacing a potential by a family of potentials, and to establish a corresponding variational principle;
\item
to discuss the existence, characterization, and number of equilibrium measures, with special attention to the new phenomena that occur in this higher-dimensional version.
\end{enumerate}
We also give a characterization of the nonlinear pressure as a Carath\'eodory dimension, which allows us to extend the notion to noncompact sets.

The most basic notion of the mathematical thermodynamic formalism is topological pressure. It was introduced by Ruelle \cite{R1} for expansive maps and by Walters \cite{W} in the general case. For a continuous map $T \colon X\to X$ on a compact metric space, the \emph{topological pressure} of a
continuous function $\phi\colon X\to\R$ is
defined by
\begin{equation}\label{ggs}
P(\phi)=\lim_{\epsilon \to 0} \limsup_{n \to\infty}
\frac1n\log\sup_E\sum_{x\in E}\exp S_n\phi(x),
\end{equation}
with the supremum taken over all $(n,\epsilon)$-separated sets~$E$ and where $S_n\phi = \sum_{k=0}^{n-1}\phi \circ T^k$.
An important relation between the topological pressure and the Kolmogorov--Sinai entropy is given by the variational principle
\begin{equation}\label{tytuAA}
P(\phi)=\sup_\mu\left(h_\mu(T)+\int_X \phi \, d
\mu\right),
\end{equation}
with the supremum taken over all $T$-invariant probability measures~$\mu$ on~$X$ and where $h_\mu(T)$ denotes the entropy with respect to~$\mu$. This was established by Ruelle \cite{R1} for expansive maps and by Walters \cite{W} in the general case.
The theory is now a broad and active independent field of study with many connections to other areas of mathematics.
We refer the reader to the books \cite{Bar11, B, KH, Ke, PPo, Pes97, Ru, Wa} for many developments.

Building on work on the Curie--Weiss mean-field theory in~\cite{BL7}, the nonlinear topological pressure was introduced in~\cite{BL20} as a generalization of \eqref{ggs} as follows (more precisely, we give an equivalent formulation using separated sets instead of covers). Given a continuous function $F\colon \R \to \R$, the \emph{nonlinear topological pressure} of a continuous function $\phi\colon X \to \R$ is given~by
\begin{equation}\label{nxx}
P_F(\phi) = \lim_{\epsilon \to 0}\limsup_{n \to \infty}\frac1{n}\log \sup_E\sum_{x \in E}\exp \left[nF\left(\frac{S_n\phi(x)}n\right)\right],
\end{equation}
with the supremum taken over all $(n,\epsilon)$-separated sets~$E$. Note that for $F(x)=x$ we recover the classical topological pressure. Buzzi and Leplaideur also established a version of the variational principle in~\eqref{tytuAA}. Namely, under an additional assumption of abundance of ergodic measures (see Section~\ref{sec2.1} for the definition), they proved that
\begin{equation}\label{hola}
P_F(\phi) = \sup_\mu\left(h_\mu(T) + F\biggl(\int_X\phi \,d\mu\biggr)\right),
\end{equation}
with the supremum taken over all $T$-invariant probability measures~$\mu$ on~$X$. In addition, they characterized the equilibrium measures of this thermodynamic formalism, that~is, the invariant probability measures at which the supremum in \eqref{hola} is attained, and they showed that a new type of phase transition can occur. Namely, one may have more than one equilibrium measure, although we still have a central limit theorem (see also~\cite{chaos, thaler}).

As described above, our main aim in the paper is to understand whether and how the results in~\cite{BL20} extend to the higher-dimensional case. This corresponds to replace the functions $F$ and~$\phi$ in~\eqref{nxx}, respectively, by a continuous function $F\colon \R^d \to \R$ and by a family $\Phi= \{\phi_1, \ldots, \phi_d\}$ of continuous functions $\phi_i\colon X \to \R$ for $i=1,\ldots,d$.
The \emph{nonlinear topological pressure} of~$\Phi$ is then defined~by
\[
P_F(\Phi) = \lim_{\epsilon \to 0}\limsup_{n \to \infty}\frac1{n}\log \sup_{E}\sum_{x \in E}\exp \left[nF\biggl(\frac{S_n\phi_1(x)}{n},\ldots, \frac{S_n\phi_d(x)}{n}\biggr)\right],
\]
with the supremum taken over all $(n,\epsilon)$-separated sets~$E$.
Whenever possible, we follow a similar approach to obtain a variational principle and to discuss the existence, characterization, and number of equilibrium measures.
We note that our setup can also be seen as a generalization of related work on the Curie--Weiss--Potts model in~\cite{BLx}.

In particular, assuming that the pair $(T, \Phi)$ has an abundance of ergodic measures, we establish the variational principle
\begin{equation}\label{hola2}
P_F(\Phi) = \sup_\mu\left(h_\mu(T) + F\biggl(\int_X\Phi \,d\mu\biggr)\right),
\end{equation}
with the supremum taken over all $T$-invariant probability measures~$\mu$ on~$X$.
As in~\cite{BL20}, for a certain class of pairs $(T,\Phi)$ we also characterize the equilibrium measures, that~is, the invariant probability measures at which the supremum in \eqref{hola2} is attained. To a moderate extent, our approach also uses some multifractal analysis results taken from~\cite{BSS02}. Consider the sets
\[
L(\Phi) = \biggl\{\biggl(\int_X\phi_1 \,d\mu,\ldots,\int_X \phi_d \,d\mu\biggr) : \mu \textrm{ is $T$-invariant}\biggr\} \subset \R^d
\]
and
\[
C_z(\Phi) = \biggl\{x \in X: \biggl(\lim_{n \to \infty}\frac{S_n\varphi_1(x)}{n},\ldots,\lim_{n \to \infty}\frac{S_n\varphi_d(x)}{n}\biggr) = z\biggr\}.
\]
We reduce the problem of finding equilibrium measures to the one of finding maximizers of the function $E\colon L(\Phi) \to \R$ defined by
\[
E(z) = h(z) + F(z),
\]
where $h\colon L(\Phi) \to \R$ is the topological entropy of $T|_{C_{z}(\Phi)}$. In fact, we show that for each $z \in \interior L(\Phi)$ maximizing $E$ there exists a unique equilibrium measure~$\nu_z$. This is actually a classical equilibrium measure for a certain function $\psi_z$ that depends on the family of functions~$\Phi$. In addition, we give conditions for the uniqueness of the equilibrium measures, both for $d=1$ and for $d>1$ (see Theorems~\ref{UD1} and~\ref{HDC}).

In contrast to the case of $d=1$, we show that there might exist points at the boundary of $L(\Phi)$ maximizing the function~$E$. For these points we do not know if there is a characterization of equilibrium measures as classical equilibrium measures.
Furthermore, we show that for $d>1$ there exist analytic pairs $(T,\Phi)$ with uncountably many equilibrium measures, in strong contrast to what happens for $d=1$ where any analytic pair $(T,\varphi)$ has only finitely many equilibrium measures.

%----------------------
\section{Nonlinear topological pressure}
%----------------------

%----------------------
\subsection{Basic notions}\label{sec2.1}
%----------------------

We first recall the notion of nonlinear topological pressure introduced by Buzzi and Leplaideur in~\cite{BL20} as an extension of the classical topological pressure.
Let $T\colon X \to X$ be a continuous map on a compact metric space $X=(X,d)$. For each $n\in\N$ we consider the distance
\[
d_n(x,y)=\max\bigl\{d(T^k(x),T^k(y)):k=0,\ldots,n-1\bigr\}.
\]
Take $n\in\N$ and $\epsilon>0$. A set $C\subset X$ is said to be an \emph{$(n,\epsilon)$-cover} of~$X$ if $\bigcup_{x \in C}B_n(x,\epsilon) = X$, where
\[
B_n(x,\epsilon) = \bigl\{y \in X: d_n(y,x) < \epsilon\bigr\}.
\]
Given a continuous function $F\colon \R \to \R$, the \emph{nonlinear topological pressure} of a continuous function $\phi\colon X \to \R$ is defined by
\[
P_F(\phi) = \lim_{\epsilon \to 0}\limsup_{n \to \infty}\frac1{n}\log \inf_{C}\sum_{x \in C}\exp \left[nF\biggl(\frac{S_n\phi(x)}{n}\biggr)\right],
\]
where $S_n\phi = \sum_{k=0}^{n-1}\phi \circ T^k$, with the infimum taken over all $(n,\epsilon)$-covers~$C$.

Let $\cM$ be the set of $T$-invariant probability measures on~$X$.
Following~\cite{BL20}, we say that the pair $(T,\phi)$ has an \emph{abundance of ergodic measures} if for each $\mu\in\cM$ and $\epsilon>0$ there exists an ergodic measure $\nu\in\cM$ such that
\[
h_{\nu}(T) \ge h_\mu(T) - \epsilon\quad\text{and}\quad
\biggl\lvert\int_X\phi \,d\nu - \int_X\phi \,d\mu\biggr\rvert < \epsilon.
\]
Assuming that $(T,\phi)$ has an abundance of ergodic measures, they obtained the variational principle
\[
P_F(\phi) = \sup_{\mu \in \cM}\biggl\{h_\mu(T) + F\biggl(\int_X\phi \,d\mu\biggr)\biggr\}.
\]
We say that $\nu \in \cM$ is an \emph{equilibrium measure for $(F,\phi)$} with respect to~$T$ if
\[
P_F(\phi) = h_{\nu}(T) + F\biggl(\int_X\phi \,d\nu\biggr).
\]

In this paper we consider a higher-dimensional generalization of the nonlinear topological pressure. Given $n\in\N$ and $\epsilon > 0$, a set $E \subset X$ is said to be $(n,\epsilon)$-\emph{separated} if $d_{n}(x,y) > \epsilon$ for every $x, y \in E$ with $x \ne y$.
Since $X$~is compact, any $(n,\epsilon)$-separated set has finite cardinality. Let $F\colon \R^d \to \R$ be a continuous function and let $\Phi= \{\phi_1, \ldots, \phi_d\}$ be a family of continuous functions $\phi_i\colon X \to \R$ for $i=1,\ldots, d$.
The \emph{nonlinear topological pressure} of the family $\Phi$ is defined~by
\[
P_F(\Phi) = \lim_{\epsilon \to 0}\limsup_{n \to \infty}\frac1{n}\log \sup_{E}\sum_{x \in E}\exp \left[nF\biggl(\frac{S_n\phi_1(x)}{n},\ldots, \frac{S_n\phi_d(x)}{n}\biggr)\right],
\]
with the supremum taken over all $(n,\epsilon)$-separated sets~$E$.
One can easily verify that the function
\[
\epsilon \mapsto \limsup_{n \to \infty}\frac1{n}\log \sup_{E}\sum_{x \in E}\exp \left[nF\biggl(\frac{S_n\phi_1(x)}{n},\ldots, \frac{S_n\phi_d(x)}{n}\biggr)\right]
\]
is nondecreasing and so $P_F(\Phi)$ is well defined.

We also describe briefly a characterization of the nonlinear topological pressure using $(n,\epsilon)$-covers.	
Let
\[
\cW_n(C)= \sum_{x \in C}\exp \left[nF\biggl(\frac{S_n\phi_1(x)}{n}, \ldots,\frac{S_n\phi_d(x)}{n}\biggr)\right].
\]
Following closely arguments in~\cite{Bar11}, one can show that
\begin{equation}\label{xxx}
P_F(\Phi) = \lim_{\epsilon \to 0}\limsup_{n \to \infty}\frac1{n}\log \inf_{C}\cW_n(C) = \lim_{\epsilon \to 0}\liminf_{n \to \infty}\frac1{n}\log \inf_{C}\cW_n(C),
\end{equation}
with the infimum taken over all $(n,\epsilon)$-covers $C$ of~$X$.

%----------------------
\subsection{Extension to noncompact sets}
%----------------------

Based on work of Pesin and Pitskel' in~\cite{PP84}, we give a characterization of the nonlinear topological pressure as a Carath\'eodory dimension. This allows one to extend the notion to noncompact sets.

We continue to consider a continuous map $T\colon X \to X$ on a compact metric space. Given a finite open cover $\cU$ of~$X$, for each $n \in \N$ let $\cX_n$ be the set of strings $U = (U_1,\ldots, U_n)$ with $U_i \in \cU$ for $i=1,\ldots,n$. We write $l(U) = n$ and we define
\[
X(U)= \bigl\{x \in X: T^{k-1} \in U_k \text{ for } k = 1,\ldots,l(U) \bigr\}.
\]
We say that $\Gamma \subset \bigcup_{n \in \N}\cX_n$ \emph{covers} a set $Z \subset X$ if $Z \subset \bigcup_{U \in \Gamma}X(U)$.

Given a family of continuous functions $\Phi = \{\phi_1,\ldots, \phi_d\}$, for each $n \in \N$ we define $S_n\Phi = (S_n\phi_1,\ldots, S_n\phi_d)$. Moreover, given a function $F\colon \R^d \to \R$, for each $U \in \cX_n$ let
\[
F_{\Phi}(U) = \begin{cases}
\sup_{X(U)} n F\left(\frac1{n}S_n\Phi \right) & \text{if } X(U) \ne \emptyset, \\ -\infty & \text{if } X(U) = \emptyset. \end{cases}
\]
Finally, given a set $Z \subset X$ and a number $\alpha \in \R$, we define
\[
M_Z(\alpha, \Phi, \cU)= \lim_{n \to \infty} \inf_{\Gamma}\sum_{U \in \Gamma}\exp (-\alpha l(U) + F_{\Phi}(U)),
\]
with the infimum taken over all $\Gamma \subset \bigcup_{k\ge n}\cX_k$ covering~$Z$ and with the convention that $\exp (-\infty) = 0$. One can easily verify that the map $\alpha \mapsto M_Z(\alpha, \Phi, \cU)$ goes from $+\infty$ to zero at a unique $\alpha \in \R$ and so one can define
\[
P_F(Z,\Phi, \cU) = \inf\bigl\{\alpha \in \R: M_Z(\alpha, \Phi, \cU) = 0\bigr\}.
\]
One can proceed as in the proof of Theorem~2.2.1 in~\cite{Bar11} to show that the limit $P_F(Z,\Phi)= \lim_{\diam \cU \to 0} P_F(Z,\Phi, \cU)$ exists.

\begin{theorem}
We have $P_F(\Phi) = P_F(X,\Phi)$.
\end{theorem}

\begin{proof}
The proof is obtained modifying arguments in Section~4.2.3 of~\cite{Bar11} and so we only give a brief sketch. Given a finite open cover $\cU$ of~$X$, we define
\[
Z_n(\Phi,\cU)= \inf_{\Gamma}\sum_{U \in \Gamma}\exp F_{\Phi}(U),
\]
with the infimum taken over all $\Gamma \subset \cX_n$ covering~$X$. Given $\Gamma_1 \subset \cX_{n_1}$ and $\Gamma_2 \subset \cX_{n_2}$, let
\[
\Gamma' = \bigl\{UV: U \in \Gamma_1 \text{ and } V \in \Gamma_2\bigr\}.
\]
Note that if $\Gamma_1$ and $\Gamma_2$ cover~$X$, then $\Gamma'$ also covers~$X$. Moreover,
\[
F_{\Phi}(UV) \le F_{\Phi}(U) + F_{\Phi}(V)
\]
for each $UV \in \Gamma'$. We have
\[
\begin{split}
Z_{n_1 + n_2}(\Phi, \cU) &\le \sum_{UV \in \Gamma'}\exp F_{\Phi}(UV)\\
& \le \sum_{U \in \Gamma_1}\exp F_{\Phi}(U) \sum_{V \in \Gamma_2}\exp F_{\Phi}(V)
\end{split}
\]
and so
\[
Z_{n_1+n_2}(\Phi,\cU) \le Z_{n_1}(\Phi,\cU)Z_{n_2}(\Phi,\cU).
\]
Therefore, one can define
\[
Z(\Phi,\cU)= \lim_{n \to \infty}\frac1{n}\log Z_n(\Phi,\cU).
\]
Finally, it follows as in
Lemmas~2.2.5 and~2.2.6 in~\cite{Bar11} that
\[
\lim_{\diam \cU \to 0}Z(\Phi,\cU) = P_F(X,\Phi)
\]
and
\[
P_F(\Phi) = \lim_{\diam \cU \to 0}Z(\Phi,\cU).
\]
This yields the desired result.
\end{proof}

%----------------------
\section{Variational Principle}
%----------------------

In this section we establish a variational principle for the nonlinear topological pressure.

Let $T\colon X \to X$ be a continuous map on a compact metric space and let $\Phi= \{\phi_1,\ldots,\phi_d\}$ be a family of continuous functions.
We say that the pair $(T,\Phi)$ has an \emph{abundance of ergodic measures} if for each $\mu\in\cM$ and $\epsilon>0$ there exists an ergodic measure $\nu\in\cM$ such that $h_{\nu}(T) \ge h_\mu(T) - \epsilon$ and
\[
\biggl\lvert \int_X\phi_i \,d\nu - \int_X\phi_i \,d\mu\biggr\rvert < \epsilon\quad\text{for} \ i =1,\ldots,d.
\]
Moreover, we say that $T$ has \emph{entropy density of ergodic measures} if for every $\mu \in \cM$ there exist ergodic measures $\nu_n\in \cM$ for $n\in\N$ such that $\nu_n \to \mu$ in the weak$^*$ topology and $h_{\nu_n}(T) \to h_\mu(T)$ when $n\to\infty$.
Note that if $T$ has entropy density of ergodic measures, then the pair $(T,\Phi)$ has an abundance of ergodic measures for any family of continuous functions~$\Phi$.

In order to give plenty examples of abundance of ergodic measures we first recall a few notions.
Given $\delta>0$, we say that $T$~has \emph{weak specification at scale~$\delta$} if there exists $\tau \in \N$ such that for every $(x_1,n_1),\ldots,(x_k,n_k) \in X \times \N$ there are $y \in X$ and times $\tau_1,\ldots, \tau_{k-1} \in \N$ such that $\tau_i \le \tau$ and
\[
d_{n_i}(T^{s_{i-1} + \tau_{i-1}}(y),x_i) < \delta \quad \text{for} \ i=1,\ldots, k,
\]
where $s_i = \sum_{i=1}^{i}n_i + \sum_{i=1}^{i-1}\tau_i$ with $n_0 = \tau_0 = 0$.
When one can take $\tau_i = \tau$ for $i =1,\ldots,k-1$, we say that $T$~has \emph{specification at scale~$\delta$}. Finally, we say that $T$~has \emph{weak specification} if it has weak specification at every scale~$\delta$ and, analogously, we say that $T$~has \emph{specification} if it has specification at every scale~$\delta$.

It was shown in~\cite{CLT20} that a continuous map $T\colon X \to X$ on a compact metric space with the weak specification property such that the entropy map $\mu \mapsto h_\mu(T)$ is upper semicontinuous, has entropy density of ergodic measures. In particular, this implies that the pair $(T,\Phi)$ has abundance of ergodic measures for any family of continuous functions~$\Phi$. Some examples of maps with abundance of ergodic measures include expansive maps with specification or with weak specification, topologically transitive locally maximal hyperbolic sets for diffeomorphisms, and transitive topological Markov chains.

The following theorem establishes a variational principle for the nonlinear topological pressure.

\begin{theorem}\label{VP}
Let $T\colon X \to X$ be a continuous map on a compact metric space and let $\Phi=\{\phi_1,\ldots, \phi_d\}$ be a family of continuous functions. Given a continuous function $F\colon \R^d \to \R$, if the pair $(T,\Phi)$ has an abundance of ergodic measures, then
\begin{equation}\label{vv3}
P_F(\Phi) = \sup_{\mu \in \cM}\biggl\{h_\mu(T) + F\biggl(\int_X\Phi \,d\mu\biggr) \biggr\},
\end{equation}
where $\int_X\Phi \,d\mu = \big(\int_X\phi_1 \,d\mu,\ldots, \int_X\phi_{d} \,d\mu\big)$.
\end{theorem}

\begin{proof}
To the possible extent we follow arguments in~\cite{BL20} for a single function. We divide the proof into two lemmas.

\begin{lemma}\label{PTF>}
We have
\[
P_F(\Phi) \ge \sup_{\mu \in \cM}\biggl\{h_\mu(T) + F\biggl(\int_X\Phi \,d\mu\biggr)\biggr\}.
\]
\end{lemma}

\begin{proof}[Proof of the lemma]
Given $r>0$, since $X$ is compact there exist $\delta,\epsilon > 0$ such that
\[
|\phi_{i}(x) - \phi_{i}(y)|<\delta/2 \quad \text{whenever} \ d(x,y) < \epsilon
\]
for $i=1,\ldots,d$ and
\[
|F(v) - F(w)| < r \quad \text{whenever} \ \|v-w\| < \delta.
\]
For definiteness we shall take the $\ell^\infty$ norm on~$\R^d$.
Now let $\mu \in \cM$ be an ergodic measure. By Birkhoff's ergodic theorem and the Shannon--McMillan--Breiman theorem, together with Egorov's theorem, there exist a set $A\subset X$ of measure $\mu(A) > 1- r$ and $N \in\N$ such that
\begin{equation}\label{SNM}
\biggl\lvert \frac{S_n\phi_i(x)}n - \int_X\phi_i \,d\mu \biggr\rvert <\delta/2
\end{equation}
for all $i=1,\ldots,d$ and
\begin{equation}\label{TTA}
\biggl\lvert \frac1{n}\log \mu(B_n(x,2\epsilon)) + h_\mu(T) \biggr\rvert < r,
\end{equation}
for $x \in A$ and $n > N$.

Now let $C$ be an $(n,\epsilon)$-cover with $\cW_n(C)$ minimal and let $D \subset C$ be a minimal $(n,\epsilon)$-cover of~$A$.
For each $x \in D$, the ball $B_n(x,\epsilon)$ intersects~$A$ at some point~$y$ (otherwise one could discard the point $x$ in~$D$). Note that
\[
d(T^k(x),T^k(y))<\epsilon\quad\text{for} \ k=0,\ldots,n-1.
\]
Hence, it follows from \eqref{SNM} that
\[
\begin{split}
\biggl\lvert \frac{S_n\phi_{i}(x)}n - \int_X\phi_{i} \,d\mu\biggr\rvert &\le \frac1n\lvert S_n\phi_{i}(x) - S_n\phi_{i}(y)\rvert \\
&\phantom{\le}+ \biggl\lvert \frac{S_n\phi_{i}(y)}n - \int_X\phi_{i} \,d\mu\biggr\rvert
< \delta/2 + \delta/2 = \delta
\end{split}
\]
for $i =1,\ldots,d$ and so
\[
\biggl\lvert F\biggl(\frac{S_n\phi_1(x)}{n},\ldots,\frac{S_n\phi_d(x)}{n}\biggr)- F\biggl(\int_X\Phi \,d\mu\biggr)\biggr\rvert < r.
\]
Moreover, $B_n(x,\epsilon) \subset B_n(y,2\epsilon)$ and so it follows from \eqref{TTA} that
\[
1-r< \mu(A) \le |D|\max_{x \in D}\mu(B_n(x,\epsilon)) \le |D|e^{-n(h_\mu(T) - r)},
\]
where $|D|$ denotes the cardinality of~$D$.
Therefore,
\[
\begin{split}
\cW_n(C) &\ge |D|\exp\bigg[nF\biggl(\int_X\Phi \,d\mu\biggr) - r\bigg]\\
&\ge (1-r)\exp[n(h_\mu(T) - r)] \exp\bigg[nF\biggl(\int_X\Phi \,d\mu\biggr) - r\bigg]
\end{split}
\]
for any sufficiently large $n \in \N$.
It follows from \eqref{xxx} that
\[
P_F(\Phi) \ge h_\mu(T) + F\biggl(\int_X\Phi \,d\mu \biggr) - 2r.
\]
Finally, by the arbitrariness of $r > 0$ we obtain
\begin{equation}\label{EGD2x}
P_F(\Phi) \ge h_\mu(T) + F\biggl(\int_X\Phi \,d\mu \biggr).
\end{equation}

Now we consider an arbitrary measure $\nu \in \cM$. Since $(T,\Phi)$ has an abundance of ergodic measures and $F$~is continuous, for each $\epsilon > 0$ there exists an ergodic measure $\mu \in \cM$ such that
\[
\biggl\lvert F\biggl(\int_X\Phi \,d\nu\biggr) - F\biggl(\int_X\Phi \,d\mu\biggr)\biggr\rvert < \epsilon \quad \text{and} \quad h_\mu(T) \ge h_{\nu}(T) - \epsilon.
\]
By \eqref{EGD2x} we have
\begin{equation}\label{AEM}
P_F(\Phi) \ge h_\mu(T) + F\biggl(\int_X\Phi \,d\mu\biggr) \ge h_{\nu}(T) + F\biggl(\int_X\Phi \,d\nu\biggr) - 2\epsilon
\end{equation}
and the desired result follows from the arbitrariness of~$\epsilon$.
\end{proof}

We also obtain the reverse inequality.

\begin{lemma}\label{PFT<}
We have
\[
P_F(\Phi) \le \sup_{\mu \in \cM}\biggl\{h_\mu(T) + F\biggl(\int_X\Phi \,d\mu \biggr)\biggr\}.
\]
\end{lemma}

\begin{proof}[Proof of the lemma]
Given $r > 0$, take $\epsilon > 0$ such that
\[
\limsup_{n \to \infty}\frac1{n}\log \inf_{C}\cW_n(C) > P_F(\Phi) - r,
\]
with the infimum taken over all $(n,\epsilon)$-covers~$C$.
Since each $(n,\epsilon)$-separated set $E_n$ is an $(n,\epsilon)$-cover, we have
\[
\limsup_{n \to \infty}\frac1{n}\log \cW_n(E_n) > P_F(\Phi) - r
\]
and there exists a diverging subsequence $(n_k)_{k \in\N}$ such that
\begin{equation}\label{WNK}
\cW_{n_k}(E_{n_{k}}) \ge \exp [n_k (P_F(\Phi) - 2r)]\quad\text{for} \ k\in\N.
\end{equation}

We cover the compact set $\Phi(X)$ by balls $B(z_i, r_i)$ for $i=1,\ldots, L$ such that $|F(z)-F(z_i)|<r$ for all $z\in B(z_i, r_i)$ and $i=1,\ldots, L$.
Now let
\[
\Lambda^{i}_{k} = \biggl\{x \in E_{n_k}: \biggl(\frac{S_{n_k}\phi_1(x)}{n_k},\ldots, \frac{S_{n_k}\phi_{d}(x)}{n_k}\biggr) \in B(z_i,r_i)\biggr\}.
\]
Note that
\[
\cW_{n_k}(E_{n_k}) \le \sum_{i=1}^{L}\cW_{n_k}(\Lambda^{i}_{k}) \le L\cW_{n_k}(\Lambda^{i}_{k}) \quad \text{for some $i \in \{1,\ldots,L\}$}
\]
and so it follows from \eqref{WNK} that
\[
\begin{split}
\exp [n_k(P_F(\Phi) - 2r)]
& \le \cW_{n_k}(E_{n_k})\\
& \le L\cW_{n_k}(\Lambda^i_k) \le L|\Lambda^i_k|\exp [n_k (F(z_i) + r)].
\end{split}
\]
This implies that
\begin{equation}\label{LNK}
|\Lambda^i_k| \ge \exp [n_k(P_F(\Phi) - F(z_i) -4r)]
\end{equation}
for any sufficiently large~$k$.
Proceeding as in the proof of the variational principle in~\cite{Mis76}, we also consider the measures
\[
\mu^i_k= \frac1{|\Lambda^i_k|}\sum_{x \in \Lambda^i_k}\delta_{x} \quad \text{and} \quad \nu^i_k= \frac1{n_k}\sum_{j=0}^{n_k - 1}\mu^i_k T^{-j}.
\]
Without loss of generality, one can assume that $\nu^i_k$ converges to a $T$-invariant measure $\mu^i$ in the weak$^*$ topology satisfying
\begin{equation}\label{HMT}
h_{\mu^i}(T) \ge \limsup_{n_k \to \infty}\frac1{n_k}\log|\Lambda^i_k|.
\end{equation}
By the definition of $\nu^i_k$ we have
\[
\begin{split}
\int_X\Phi \,d\mu^i
&= \lim_{k \to \infty}\int_X\Phi \,d\nu^i_k \\
&=\lim_{k \to \infty}\biggl(\int_X\frac{S_{n_k}\phi_1}{n_k} \,d\mu^i_k,\ldots,\int_X\frac{S_{n_k} \phi_{d}}{n_k} \,d\mu^i_k\biggr) \in \overline{B(z_i,r_i)}.
\end{split}
\]
Hence, by \eqref{LNK} and \eqref{HMT} we obtain
\[
\begin{split}
h_{\mu^i}(T) + F\biggl(\int_X\Phi \,d\mu^i \biggr) &\ge P_F(\Phi) - F(z_i) -4r + F(z_i) - r\\
& = P_F(\Phi)- 5r.
\end{split}
\]
The desired result follows from the arbitrariness of~$r$.
\end{proof}

Lemmas~\ref{PTF>} and~\ref{PFT<} establish the statement in the theorem.
\end{proof}

For a general continuous map~$T$, we obtain a variational principle for an arbitrary convex function~$F$.

\begin{theorem}\label{VPC}
Let $T\colon X \to X$ be a continuous map on a compact metric space and let $\Phi=\{\phi_1,\ldots, \phi_d\}$ be a family of continuous functions. If $F\colon \R^d \to \R$ is a convex continuous function, then identity \eqref{vv3} holds.
\end{theorem}

\begin{proof}
It follows from the proof of Lemma~\ref{PTF>} (see~\eqref{EGD2x}) that
\[
P_F(\Phi) \ge h_\mu(T) + F\biggl(\int_X\Phi \,d\mu \biggr)
\]
for every ergodic measure $\mu \in \cM$.
Now let $\nu \in \cM$ be an arbitrary measure and consider its ergodic decomposition with respect to~$T$. It is described by a probability measure $\tau$ on~$\cM$ that is concentrated on the subset of ergodic measures $\cMe$. We recall that for every bounded measurable function $\psi\colon X \to \R$ we have
\[
\int_X \psi \,d\nu = \int_\cM \biggl(\int_X \psi \,d\mu \biggr) \,d\tau(\mu).
\]

For a convex function $F$ one can use Jensen's inequality to obtain
\[
\begin{split}
F\biggl(\int_X\Phi \,d\nu \biggr) &= F\biggl(\int_\cM\biggl(\int_X\phi_1 \,d\mu\biggr) d\tau(\mu),\ldots,\int_\cM\biggl( \int_X\phi_d \,d\mu\biggr)d\tau(\mu)\biggr)\\
& \le \int_\cM F\biggl(\int_X\Phi \,d\mu\biggr)d\tau(\mu).
\end{split}
\]
Hence,
\begin{equation}\label{CVX}
h_{\nu}(T) + F\biggl(\int_X\Phi \,d\nu \biggr) \le \int_\cM\biggl[h_\mu(T) + F\biggl(\int_X\Phi \,d\mu \biggr)\biggr]d\tau(\mu)
\le P_F(\Phi).
\end{equation}
The desired result follows now readily from Lemma~\ref{PFT<} whose proof does not require an abundance of ergodic measures.
\end{proof}

We also obtain a variational principle over the ergodic measures.

\begin{corollary}
Let $T\colon X \to X$ be a continuous map on a compact metric space, let $\Phi=\{\phi_1,\ldots, \phi_d\}$ be a family of continuous functions, and let $F\colon \R^d \to \R$ be a continuous function. If the pair $(T,\Phi)$ has an abundance of ergodic measures or $F$ is convex, then
\begin{equation}\label{PTOP}
P_F(\Phi) = \sup_{\mu \in \cMe}\biggl\{h_\mu(T) + F\biggl(\int_X\Phi \,d\mu\biggr) \biggr\}.
\end{equation}
\end{corollary}

\begin{proof}
Since $\cMe \subset \cM$, we have
\[
\sup_{\mu \in \cM} \biggl\{h_\mu(T) + F\biggl(\int_X\Phi \,d\mu\biggr)\biggr\} \ge \sup_{\mu \in \cMe}\biggl\{h_\mu(T) + F\biggl(\int_X\Phi \,d\mu\biggr)\biggr\}.
\]

Now we establish the reverse inequality. Assume that the pair $(T,\Phi)$ has an abundance of ergodic measures. It follows from \eqref{AEM} that for each $\nu\in\cM$ and $\epsilon>0$, there exists an ergodic measure $\mu\in\cM$ such that
\[
h_\mu(T) + F\bigg(\int_X\Phi\, d\mu\bigg) \ge h_{\nu}(T) + F\bigg(\int_X\Phi d\nu\bigg)-\epsilon.
\]
Since $\epsilon$ is arbitrary, this readily implies that
\[
\sup_{\mu \in \cM}\biggl\{h_\mu(T) + F\biggl(\int_X\Phi \,d\mu\biggr)\biggr\} \le \sup_{\nu \in \cMe}\biggl\{h_{\nu}(T) + F\biggl(\int_X\Phi \,d\nu\biggr)\biggr\}.
\]
Finally, it follows from Theorem~\ref{VP} that identity \eqref{PTOP} holds.

Now assume that $F$ is convex. It follows from \eqref{CVX} that
\[
h_{\nu}(T) + F\biggl(\int_X\Phi \,d\nu \biggr) \le \sup_{\mu \in \cMe}\biggl\{h_\mu(T) + F\biggl(\int_X\Phi \,d\mu\biggr)\biggr\}
\]
for each $\nu\in\cM$. Therefore,
\[
\sup_{\nu \in \cM}\biggl\{h_{\nu}(T) + F\biggl(\int_X\Phi \,d\nu \biggr)\biggr\} \le \sup_{\mu \in \cMe}\biggl\{h_\mu(T) + F\biggl(\int_X\Phi \,d\mu\biggr)\biggr\}
\]
and applying Theorem \ref{VPC} we also obtain identity~\eqref{PTOP}.
\end{proof}

%----------------------
\section{Equilibrium measures: existence and characterization}
%----------------------

In this section we consider the problem of characterizing the equilibrium measures of the nonlinear topological pressure.

%----------------------
\subsection{Existence of equilibrium measures}
%----------------------

In view of Theorem~\ref{VP}, we say that $\nu \in \cM$ is an \emph{equilibrium measure for $(F,\Phi)$} with respect to~$T$ if
\[
P_F(\phi) = h_\mu(T) + F\biggl(\int_X\Phi \,d\mu\biggr).
\]
We first formulate a result on the existence of equilibrium measures.

\begin{theorem}\label{EXI}
Let $T\colon X \to X$ be a continuous map on a compact metric space such that the map $\mu \mapsto h_\mu(T)$ is upper semicontinuous, let $\Phi=\{\phi_1,\ldots, \phi_d\}$ be a family of continuous functions, and let $F\colon \R^d \to \R$ be a continuous function. If the pair $(T,\Phi)$ has an abundance of ergodic measures or $F$ is convex, then there exists at least one equilibrium measure for $(F,\Phi)$.
\end{theorem}

\begin{proof}
Since the map $\mu \mapsto h_\mu(T)$ is upper semicontinuous, $F$ is continuous and the map $\mu \mapsto \int_X\psi \,d\mu$ is continuous for each continuous function $\psi\colon X \to \R$, we conclude that $\mu \mapsto h_\mu(T) + F(\int_X\Phi \,d\mu)$ is upper semicontinuous. Together with the compactness of~$\cM$, this guarantees that there exists a measure $\mu_{\Phi} \in \cM$ such that
\[
\sup_{\mu \in \cM} \biggl\{h_\mu(T) + F\biggl(\int_X\Phi \,d\mu\biggr)\biggr\} = h_{\mu_{\Phi}}(T) + F\biggl(\int_X\Phi \,d\mu_{\Phi}\biggr).
\]
Hence, it follows from the variational principles in Theorems~\ref{VP} and~\ref{VPC} that $\mu_{\Phi}$~is an equilibrium measure for $(F,\Phi)$.
\end{proof}

In some cases one can pass to the one-dimensional setting of the nonlinear thermodynamic formalism.

\begin{example}
Consider the function $F\colon \R^d \to \R$ defined by
\[
F(z_1,\ldots,z_d) = f(\alpha_1z_1 + \cdots + \alpha_dz_d),
\]
where $f\colon \R \to \R$ is a continuous function and $\alpha_j \in \R$ for $j =1,\ldots,d$. Then
\[
F\biggl(\int_X\Phi \,d\mu\biggr) = F\biggl(\int_X\phi_1 \,d\mu,\ldots,\int_X\phi_d \,d\mu\biggr) = f\biggl(\int_X\phi \,d\mu\biggr)
\]
for every $\mu \in \cM$, where
\[
\phi=\alpha_1\phi_1 + \cdots + \alpha_d\phi_d.
\]
Moreover, $P_F(\Phi) = P_f(\phi)$ and this implies that $(F,\Phi)$ and $(f,\phi)$ have the same equilibrium measures. In other words, for a function $F$ as above the study of equilibrium measures can be reduced to the case when $d=1$.
\end{example}

Of course, in general the continuous function $F$ can be much more complicated. For instance, the Potts model involves the study of the topological pressure for $F(z_1,\ldots,z_2) = (z_1^2 + \cdots + z_d^2)^{1/2}$.

For the following example, we recall the notion of cohomology. We say that a function $\phi\colon X \to \R$ is \emph{cohomologous} to a function $\psi\colon X \to \R$ (with respect to~$T$) if there exists a measurable bounded function $q\colon X \to \R$ such~that
\[
\phi(x) = \psi(x) + q(T(x)) - q(x)\quad\text{for} \ x \in X.
\]

\begin{example}[Reduction of dimension via cohomology]
Let $T$ be a continuous map on a compact metric space and let $\Phi=\{\phi_1,\ldots, \phi_d\}$ be a family of continuous functions such that the pair $(T,\Phi)$ has an abundance of ergodic measures. Let $F\colon \R^d \to \R$ be a continuous function and assume that $\phi_1$ is cohomologous to $\phi_d$. This implies that $\int_X\phi_1 \,d\mu = \int_X\phi_{d} \,d\mu$ for every $\mu \in \cM$. Therefore,
\[
\begin{split}
F\biggl(\int_X\Phi \,d\mu\biggr) &= F\biggl(\int_X\phi_1 \,d\mu, \int_X\phi_2 \,d\mu,\ldots,\int_X\phi_1 \,d\mu\biggr) \\
&= G\biggl(\int_X\phi_1 \,d\mu,\ldots,\int_X\phi_{d-1} \,d\mu\biggr)
\end{split}
\]
for every $\mu \in \cM$, where
\[
G(z_1,\ldots,z_{d-1}) = F(z_1, z_2,\ldots,z_{d-1},z_1)
\]
for each $(z_1,\ldots,z_d) \in \R^d$. The cohomology assumption also implies that
\[
\|S_n\phi_1-S_n\phi_{d}\|_{\infty}/n \to 0\quad\text{when} \ n \to \infty.
\]
Together with the continuity of~$F$, this implies that $P_{F}(\Phi) = P_{G}(\Psi)$, where $\Psi = \{\phi_1,\ldots,\phi_{d-1}\}$. Hence, the pairs $(F,\Phi)$ and $(G,\Psi)$ have the same equilibrium measures. More generally, one could consider any cohomology relation between two or more functions in~$\Phi$ or even more cohomology relations in order to reduce further the dimension of the problem.
\end{example}

\begin{example}[Reduction to the classical case via cohomology]
Let $T$ be a continuous map on a compact metric space and let $\Phi = \{\phi_1, \phi_2\}$ be a pair of continuous functions such that $(T,\Phi)$ has an abundance of ergodic measures. Moreover, assume that $\phi_1$ is cohomologous to~$\phi_2$ and consider the function $F\colon \R^2 \to \R$ given by $F(z_1,z_2) = (z_1^3+ z_2^3)^{1/3}$. This implies that $\int_X\phi_1 \,d\mu = \int_X\phi_2 \,d\mu$ for every $\mu \in \cM$ and so
\[
\begin{split}
h_\mu(T) + F\biggl(\int_X\Phi \,d\mu\biggr) &= h_\mu(T) + F\biggl(\int_X\phi_1 \,d\mu,\int_X\phi_1 \,d\mu\biggr) \\
&= h_\mu(T) + \int_X2^{1/3}\phi_1\, d\mu
\end{split}
\]
for every $\mu \in \cM$.
Letting $\psi= 2^{1/3}\phi_1$, it follows from the definitions that $P_F(\Phi) = P(\psi)$, where $P$ denotes the classical topological pressure. Hence, $\nu$~is an equilibrium measure for $(F,\Phi)$ if and only if $\nu$~is an equilibrium measure for~$\psi$.

Recall that a continuous function $\phi\colon X \to \R$ is said to have the \emph{Bowen property} if for every $K > 0$ there exists $\epsilon > 0$ such that whenever
\[
d(T^k(x),T^k(y)) < \epsilon\quad\text{for} \ k = 0,1,\ldots,n-1
\]
we have $|S_n\phi(x) - S_n\phi(y)| \le K$.
If $T\colon X \to X$ is an expansive map with specification and $\phi_1$ (or~$\phi_2$) is a continuous function with the Bowen property, then there exists a unique equilibrium measure $\mu_\psi$ for~$\psi$ (see~\cite{Bow75}). Therefore, $\mu_\psi$ is also the unique equilibrium measure for $(F,\Phi)$.
\end{example}

We observe that this example can be easily generalized to the case when
\[
F(z_1,\ldots,z_d)^n = H_{n}(z_1,\ldots,z_d),
\]
where $H_n$ is a homogeneous polynomial of degree~$n$, together with some cohomology relations between the functions in~$\Phi$.

%----------------------
\subsection{Characterization of equilibrium measures}
%----------------------

Now we consider the problem of characterizing the equilibrium measures.
Given a pair $(T,\Phi)$, we consider the set
\[
L(\Phi)=\biggl\{\int_X\Phi \,d\mu :\mu \in \cM\biggr\}.
\]
Since the map $\mu \mapsto \int_X\psi \,d\mu$ is continuous for each continuous function $\psi\colon X\to\R$ and $\cM$ is compact and connected, the set $L(\Phi)$ is a compact and connected subset of~$\R^d$.
For each $z \in \R^d$, we also consider the level sets
\[
\cM(z)= \biggl\{\mu \in \cM: \int_X\Phi \,d\mu = z\biggr\}
\]
and
\begin{equation}\label{Cz}
C_z(\Phi)= \biggl\{x \in X: \lim_{n \to \infty}\frac{S_n\Phi(x)}{n} = z \biggr\}.
\end{equation}
Following closely~\cite{BL20}, we say that the pair $(T,\Phi)$ is $C^r$-\emph{regular} (for some $2 \le r \le \omega$, where $\omega$~refers to the analytic case) if the following holds:
\begin{enumerate}
\item
each function in $\Span\{\phi_1,\ldots,\phi_d,1 \}$ has a unique equilibrium measure with respect to the classical topological pressure;
\item
for each $z \in \interior L(\Phi)$ the map $q \mapsto P(\langle q, \Phi - z\rangle)$ is of class $C^r$, is~strictly convex, and its second derivative is a positive definite bilinear form for each $q \in \R^d$, where $\langle\cdot, \cdot \rangle$ is the usual inner product;
\item
the entropy map $\mu \mapsto h_\mu(T)$ is upper semicontinuous.
\end{enumerate}
Examples of $C^r$-regular pairs $(T,\Phi)$ include topologically mixing subshifts of finite type, $C^{1+\epsilon}$ expanding maps, and $C^{1+\epsilon}$ diffeomorphisms with a locally maximal hyperbolic set, with $\Phi$ composed of H\"older continuous functions.
Finally, we say that the family of functions $\Phi = \{\phi_1,\ldots,\phi_d\}$ is \emph{cohomologous} to a constant $c = (c_1,\ldots,c_d)$ if $\phi_{i}$ is cohomologous $c_i$ for $i=1,\ldots,d$. Then $L(\Phi) = \{c\}$ and so $\interior L(\Phi) = \emptyset$.

The following theorem is our main result. Given a function $F\colon \R^d \to \R$, we consider the set
\[
K(F,\Phi)= \biggl\{\int_X\Phi \,d\mu: \mu \text{ is an equilibrium measure for } (F,\Phi)\biggr\} \subset L(\Phi).
\]
We also consider the function $h\colon L(\Phi) \to \R$ defined by
\begin{equation}\label{hh}
h(z)= \sup\bigl\{h_\mu(T): \mu \in \cM(z)\bigr\}.
\end{equation}

\begin{theorem}\label{THM}
Let $T\colon X \to X$ be a continuous map on a compact metric space and let $\Phi=\{\phi_1,\ldots, \phi_d\}$ be a family of continuous functions such that the pair $(T,\Phi)$ is $C^r$ regular for some $r \ge 2$. For each function $F\colon \R^d \to \R$ of class $C^r$, the following properties hold:
\begin{enumerate}
\item
$K(F,\Phi)$ is a nonempty compact set;
\item
$K(F,\Phi)$ is the set of maximizers of the function $z \mapsto h(z) + F(z)$;
\item
if $K(F,\Phi) \subset \interior L(\Phi)$, then the equilibrium measures for $(F,\Phi)$ are the elements of $\{\nu_{z}: z \in K(F,\Phi)\}$, where each $\nu_z\in\cM$ is an ergodic measure that is the unique equilibrium measure for some function $\psi_z \in \Span\{\phi_1,\ldots,\phi_d,1\}$.
\end{enumerate}
\end{theorem}

\begin{proof}
We divide the proof into steps.

\begin{lemma}\label{KFP}
$K(F,\Phi)$ is a nonempty compact subset of $L(\Phi)$.
\end{lemma}

\begin{proof}[Proof of the lemma]
Let $(z_n)_{n\in\N}$ be a sequence in $K(F,\Phi)$ converging to a point $z \in L(\Phi)$. For each $n \in\N$ there exists an equilibrium measure $\mu_n \in \cM$ for $(F,\Phi)$ such that $z_n= \int_X\Phi \,d\mu_n$. Passing eventually to a subsequence, we may assume that there exists $\mu \in \cM$ such that $\mu_n \to \mu$ when $n\to\infty$ in the weak$^*$ topology. Since the map $\mu \mapsto h_\mu(T)$ is upper semicontinuous, we obtain
\[
P_F(\Phi) = \limsup_{n \to \infty}\bigg[h_{\mu_n}(T) + F\biggl(\int_X\Phi \,d\mu_n\biggr)\bigg] \le h_\mu(T) + F\biggl(\int_X\Phi \,d\mu\biggr),
\]
which implies that $\mu$ is an equilibrium measure for $(F,\Phi)$. Since $z = \int_X\Phi \,d\mu$, we conclude that $z \in K(F,\Phi)$. Hence, $K(F,\Phi)$ is a closed and bounded subset of~$\R^d$. Theorem~\ref{EXI} guarantees that $K(F,\Phi)$ is nonempty.
\end{proof}

\begin{lemma}\label{BSS}
For each $z \in \interior L(\Phi)$ there exists an ergodic measure $\nu_{z} \in \cM$ such that $\int_X\Phi \,d\nu_z = z$. In fact, $\nu_z$ is the unique equilibrium measure for some function $\psi_z \in \Span\{\phi_1,\ldots,\phi_d,1\}$.
\end{lemma}

\begin{proof}[Proof of the lemma]
By the proof of Theorem~8 in~\cite{BSS02}, for each $z \in \interior L(\Phi)$ there exists $q(z) \in \R^d$ and an ergodic measure $\nu_z \in \cM(z)$ such that $\nu_z$ is an equilibrium measure for the function
\[
\psi_z = \langle q(z), \Phi - z\rangle - h(T|_{C_z(\Phi)}),
\]
where $h$ denotes the topological entropy. Since $\psi_z \in \Span\{\phi_1,\ldots,\phi_{d},1\}$, we conclude that $\nu_z$ is the unique equilibrium measure for $\psi_z$.
\end{proof}

\begin{lemma}\label{FZ4}
For each $z \in L(\Phi)$ there exists $\mu \in \cM(z)$ with $h(z) = h_\mu (T)$. Moreover, when $z \in \interior L(\Phi)$ this measure is unique and coincides with~$\nu_z$.
\end{lemma}

\begin{proof}[Proof of the lemma]
Take $z \in L(\Phi)$. By the definition of $L(\Phi)$, there exists $\mu \in \cM$ such that $\int_X\Phi \, d\mu = z$, that is, $\cM(z) \ne \emptyset$.
By the compactness of $\cM(z)$ and the upper semicontinuity of the map $\mu \mapsto h_\mu(T)$, there exists $\mu \in \cM(z)$ maximizing the metric entropy.

Now take $z \in \interior L(\Phi)$. By Lemma~\ref{BSS}, there exists $\nu_z \in \cM$ such that $\int_X\Phi \,d\nu_z = z$, where $\nu_z$ is the unique equilibrium for~$\psi_z$. Let $\mu \in \cM(z)$ be a measure maximizing the metric entropy. Since $\int_X\Phi \,d\mu = \int_X\Phi \,d\nu_z$, one can verify that $\int_X\psi_z \,d\mu = \int_X\psi_z \,d\nu_z$. Then
\[
h_\mu(T) + \int_X\psi_z \,d\mu \ge h_{\nu_z}(T) + \int_X\psi_z \,d\nu_z = P(\psi_z),
\]	
which implies that $\mu$ is also an equilibrium measure for~$\psi_z$ (for the classical topological pressure). Since $\psi_z$ has a unique equilibrium measure, we conclude that $\mu = \nu_z$.
\end{proof}

\begin{lemma}\label{EQU}
$z \in K(F,\Phi)$ if and only if $z$ maximizes the function $E$ defined by $E(z)= h(z) + F(z)$.
\end{lemma}

\begin{proof}[Proof of the lemma]
First assume that $z \in L(\Phi)$ maximizes the function~$E$. By Lemma~\ref{FZ4}, there exists $\mu \in \cM(z)$ such that $h(z) = h_\mu(T)$ and so
\[
h_\mu(T) + F\biggl(\int_X\Phi \,d\mu\biggr) = h(z) + F(z) = \sup_{\mu \in \cM}\biggl\{h_\mu(T) + F\biggl(\int_X\Phi \,d\mu\biggr)\biggr\}.
\]
This implies that $\mu$ is an equilibrium measure for $(F,\Phi)$ and so $z \in K(F,\Phi)$.
	
Now assume that $z \in K(F,\Phi)$. Then there exists an equilibrium measure~$\mu$ for $(F,\Phi)$ such that $z = \int_X\Phi \,d\mu$ and so
\[
\begin{split}
E(z)
&= h(z) + F(z) \\
&\ge h_\mu(T) + F\biggl(\int_X\Phi \,d\mu\biggr)\\
& = \sup_{\mu \in \cM}\biggl\{h_\mu(T) + F\biggl(\int_X\Phi \,d\mu\biggr)\biggr\}.
\end{split}
\]
This shows that $z$ maximizes~$E$.
\end{proof}

Lemmas~\ref{KFP} and~\ref{EQU} give items (1) and~(2) in the theorem. Now we establish item~(3). For each $z \in K(F,\Phi)$ there exists an equilibrium measure $\mu$ for $(F,\Phi)$ such that $\int_X\Phi \,d\mu = z$. When $K(F,\Phi) \subset \interior L(\Phi)$, it follows from Lemmas~\ref{BSS} and~\ref{FZ4} that $\mu$ is the unique measure with $\int_X\Phi \,d\mu = z$ and that $\mu = \nu_z$, where $\nu_z$ is ergodic and is the unique equilibrium measure for some function~$\psi_z$.
\end{proof}

%----------------------
\section{Number of equilibrium measures}
%----------------------

In this section we consider the problem of how many equilibrium measures a $C^r$ regular system has.

%----------------------
\subsection{Preliminary results}
%----------------------

We start with some auxiliary results about the function $h$ in~\eqref{hh}. Note that
\[
h(z)=\sup\left\{h_\mu(T):\int_X\Phi \,d\mu = z\text{ with $\mu \in \cM$}\right\}.
\]

\begin{proposition}\label{HLP}
Let $(T, \Phi)$ be a $C^r$ regular pair for some $r \ge 2$. Then the function $h\colon L(\Phi) \to \R$ is upper semicontinuous, concave and finite.
\end{proposition}

\begin{proof}
Take $z \in L(\Phi)$ and consider a sequence $(z_n)_{n \in\N}$ in $L(\Phi)$ such that $z_n \to z$ when $n\to\infty$. By Lemma~\ref{FZ4}, eventually passing to a subsequence one can assume that for each $n \in\N$ there exists $\mu_n \in \cM(z_n)$ such that $h(z_n) = h_{\mu_n}(T)$ and $\mu_n \to \mu$ when $n\to\infty$ for some $\mu \in \cM$ in the weak$^*$ topology. We also have
\[
\int_X\Phi \,d\mu = \lim_{n \to \infty}\int_X\Phi \,d\mu_{n} = \lim_{n \to \infty}z_n = z
\]
and so $\mu \in \cM(z)$.
Moreover, since $\mu \mapsto h_\mu(T)$ is upper semicontinuous, we obtain
\[
\limsup_{n \to \infty}h(z_n) = \limsup_{n \to \infty}h_{\mu_n}(T) \le h_\mu(T) \le h(z)
\]	
and so $h$ is upper semicontinuous on~$L(\Phi)$.

Now we prove the concavity property. Take $z_1, z_2 \in L(\Phi)$ and $\mu_1 \in \cM(z_1)$, $\mu_2 \in \cM(z_2)$ such that $h(z_1) = h_{\mu_1}(T)$ and $h(z_2) = h_{\mu_2}(T)$. Since the entropy map is affine, for each $t \in [0,1]$ we have
\[
\begin{split}
h(tz_1 + (1-t)z_2) \ge h_{t\mu_1 + (1-t)\mu_2}(T) &= th_{\mu_1}(T) + (1-t)h_{\mu_2}(T)\\
& = th(z_1) + (1-t)h(z_2).
\end{split}
\]
The upper semicontinuity of $h$ on $L(\Phi)$ together with the compactness of $L(\Phi)$ and the fact that $\cM(z) \ne \emptyset$ for each $z \in L(\Phi)$, guarantee that $h$ is finite on $L(\Phi)$.
\end{proof}

As pointed out in the recent work~\cite{Wol20}, in strong contrast to what happens for $d=1$, the function $z \mapsto h(z)$ need not be continuous on~$L(\Phi)$.

\begin{proposition}\label{CRW}
If the pair $(T,\Phi)$ is $C^r$ regular, then the function $h|_{\interior L(\Phi)}$ is $C^{r-1}$. Moreover, if $(T,\Phi)$ is $C^{\omega}$ regular, then $h|_{\interior L(\Phi)}$ is analytic.
\end{proposition}

\begin{proof}
It follows from Theorem~12 in~\cite{BSS02} that if $(T,\Phi)$ is $C^r$ regular, then the map $\interior L(\Phi) \ni z \mapsto h(T|_{C_z(\Phi)})$ is of class $C^{r-1}$, and that if the pair is $C^{\omega}$ regular, then this map is analytic. Since $h(z) = h(T|_{C_z(\Phi)})$ for every $z \in \interior L(\Phi)$, we obtain the desired statement.
\end{proof}

For $d=1$, Corollary~1.11 in~\cite{BL20} says that if the pair $(T,\Phi)$ is $C^\omega$ and $F$~is analytic on $\interior L(\Phi)$, then the set $K(F,\Phi)$ is finite. In particular, there exist finitely many equilibrium measures. The next example shows that this may not hold for $d>1$.

\begin{example}[Infinite number of equilibrium measures]\label{exa5}
Let $(T,\Phi)$ be a two-dimensional $C^{\omega}$ regular pair. Proposition~\ref{CRW} says that the function $h|_{\interior L(\Phi)}$ is analytic. Now we consider the function $F\colon \interior L(\Phi) \to \R$ given by
\[
F(z_1,z_2) = - h(z_1,z_2) - z_1^2 z_2^2.
\]
Since $h$ is analytic on $\interior L(\Phi)$, the same happens to~$F$. Notice that
\[
E(z_1,z_2) = h(z_1,z_2) + F(z_1,z_2) = - z_1^2 z_2^2.
\]
One can easily verify that the set of critical points of $E$ in $\interior L(\Phi)$ is
\[
\{(x,0): x \in \R\} \cup \{(0,y): y \in \R\} \cap \interior L(\Phi).
\]
In fact, all critical points maximize $E$ on~$\interior L(\Phi)$. Now assume that $(0,0)$ belongs to~$\interior L(\Phi)$. Then any open neighborhood of $(0,0)$ in $\interior L(\Phi)$ contains an uncountable number of critical points. In particular, by Theorem~\ref{THM}, there exist infinitely many equilibrium measures for~$(F, \Phi)$.
\end{example}

Example~\ref{exa5} also shows that one can have $K(F,\Phi) \cap \partial L(\Phi) \ne \emptyset$. In fact, if $(0,0) \in \interior L(\Phi)$, then $\partial L(\Phi)$ intersects the two axes at least four times. This justifies requiring that $K(F,\Phi) \subset \interior L(\Phi)$ in the last item of Theorem~\ref{THM}.

%----------------------
\subsection{Equilibrium measures I}
%----------------------

For $d=1$, it was shown in~\cite{BL20} that no point on $\partial L(\Phi)$ maximizes the function $E = h + F$. By Lemma~\ref{EQU}, this implies that $K(F,\Phi) \subset \interior L(\Phi)$. It is also shown that $h''(z) < 0$ for every $z \in \interior L(\Phi)$ and so $h\colon L(\Phi) \to \R$ is a strictly concave function. Note that for $d = 1$ we have $L(\phi) = [A, B]$, where $A = \inf_{\mu \in \cM}\int \phi \,d\mu$ and $B = \sup_{\mu \in \cM}\int \phi \,d\mu$.

The next result is a criterion for uniqueness of equilibrium measures.

\begin{theorem}\label{UD1}
Let $(T,\phi)$ be a $C^r$ regular pair and let $F\colon [A,B] \to \R$ be a concave $C^r$ function. Then there exists a unique equilibrium measure for $(F,\phi)$. Moreover, the equilibrium measure is ergodic.
\end{theorem}

\begin{proof}
Since $F$ is concave and $h$ is strictly concave, the function $E = h + F$ is strictly concave. This implies that $E$ has at most one maximizer in $(A,B)$. Since there is no maximizer of~$E$ on $\partial L(\phi) = \{A,B\}$ and $K(F,\phi) \ne \emptyset$, we conclude that there exists a unique point $z^* \in (A,B)$ maximizing $E$. Hence, it follows from Lemma~\ref{EQU} that $K(F,\phi) = \{z^*\}$. By Theorem~4.3 in~\cite{BL20} together with Lemma~\ref{BSS}, we conclude that there exists a unique equilibrium measure for $(F,\phi)$ and that this measure is ergodic.
\end{proof}

The following example illustrates various possibilities.

\begin{example}\label{exa6}
Let $\Sigma = \{-1,1\}$ and let $T\colon \Sigma^{\Z} \to \Sigma^{\Z}$ be the two-sided shift. We consider the function $\phi\colon \Sigma \to \R$ defined by $\phi(\cdots \omega_{-1} \omega_0 \omega_1 \cdots) = \omega_0$. Then $L(\phi) = [-1,1]$ and the entropy function $h\colon L(\phi) \to \R$ is given by
\begin{equation}\label{SPEC}
h(z) = -\frac{1-z}2\log \left(\frac{1-z}2\right) -\frac{1+z}2\log \left(\frac{1+z}2\right).
\end{equation}
For the function $F\colon L(\Phi) \to \R$ defined by $F(z) = \alpha/(z^2 - 2)$, where $\alpha \in \R$, we have
\[
F''(z) = 2\alpha(3z^2 + 2)/(z^2 - 2)^3.
\]
Notice that for $\alpha > 0$ we have $F''< 0$ on $\interior L(\phi)$. Since $F \equiv 0$ for $\alpha = 0$, the function $F$ is concave on $\interior L(\phi)$ whenever $\alpha \ge 0$. Hence, by Theorem~\ref{UD1} there exists a unique equilibrium measure $\nu_{z^*}$ for $(F, \phi)$, where $z^* = 0$ (see Figure~\ref{fig:S1}). For $\alpha < 0$, the number of equilibrium measures may vary. For instance, for $\alpha = -1$ there is one equilibrium measure, while for $\alpha = -1.7$ there are two equilibrium measures.
\end{example}

\begin{figure}[htbp]
\centering
\includegraphics[width=0.6\textwidth]{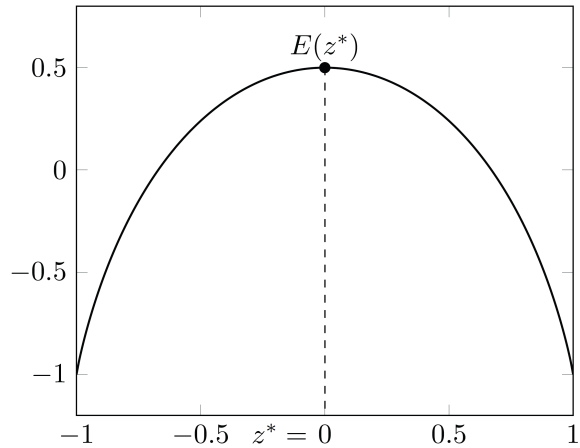}
\caption{Graph of $E$ for $\alpha = 1$.}\label{fig:S1}
\end{figure}

%\begin{figure}
%\begin{tikzpicture}[scale=1.0]
%\begin{axis}[
%xmin = -1, xmax = 1,
%ymin = -1.2, ymax = 0.8]]
%\addplot[
%domain = -1:1,
%samples = 200,
%smooth,
%thick,
%black,
%] {-((1+x)/2)*log2((1+x)/2) - ((1-x)/2)*log2((1-x)/2) + 1/(x*x -2)};
%\node (x1) at (100,180) {$E(z^*)$};
%\filldraw [black] (100,170) circle (2pt);
%\draw [black, dashed] (100,170) -- (100,0);
%\end{axis}
%\node (x2) at (2.8,-0.25) {$z^* = $};
%\end{tikzpicture}
%\caption{Graph of $E$ for $\alpha = 1$.} \label{fig:S1}
%\end{figure}

Theorem~\ref{UD1} also shows that in order to have finitely many equilibrium measures it is not necessary that the pair $(T, \phi)$ is $C^\omega$ and that the function $F$~is analytic. We give an example in the nonanalytic $C^{\infty}$ case.

\begin{example}
Consider the pair $(T,\phi)$ in Example~\ref{exa6} and let $F\colon \R \to \R$ be the function given by
\[
F(z) = \begin{cases} 3\exp(-1/z) & \text{if } z > 0, \\
0 & \text{if } z \le 0. \end{cases}
\]
One can show that $F$ is $C^\infty$ but not analytic. For $-1 \le z \le 0$, we~have
\[
E = h + F = h + 0 = h.
\]
It follows from \eqref{SPEC} that $E$ has a local maximum $y_1 = 1$ at $z_1^{*} = 0$. For $0 < z \le 1$, one can verify that $E$ has a local maximum $y_2 \approx 1.33$ at $z_2^{*} \approx 0.75$. Since $y_1 < y_2$, the function $E$ has a unique global maximum at $z_2^{*} \in (0, 1) \subset \interior L(\phi)$ (see Figure~\ref{fig:S2}). By Theorem~\ref{THM}, we conclude that $\nu_{z_2^*}$ is the unique equilibrium measure for $(F,\phi)$.
\end{example}

\begin{figure}[htbp]
\centering
\includegraphics[width=0.6\textwidth]{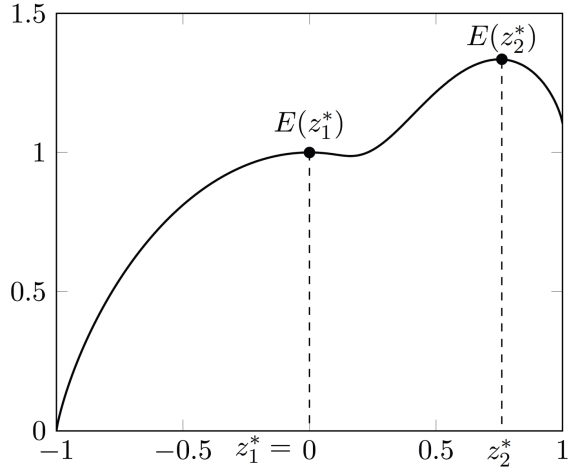}
\caption{Nonanalytic $C^{\infty}$ case}\label{fig:S2}
\end{figure}

%\begin{figure}
%\begin{tikzpicture}[scale=1.0]
%\begin{axis}[enlargelimits=false,xmin = -1, xmax = 1, ymin = 0, ymax = 1.5]
%\addplot[domain = -1:0,samples = 200,smooth, thick]{-((1+x)/2)*log2((1+x)/2) - ((1-x)/2)*log2((1-x)/2)};
%\addplot[domain = 0:1,samples = 200,smooth, thick]{-((1+x)/2)*log2((1+x)/2) - ((1-x)/2)*log2((1-x)/2) + 3*exp(-1/x)};
%\node (x1) at (100,110) {$E(z_1^*)$};
%\draw [black, dashed] (100,100) -- (100,0);
%\filldraw [black] (100,100) circle (2pt);
%\node (x2) at (176,141) {$E(z_2^*)$};
%\filldraw [black] (176,133.5) circle (2pt);
%\draw [black, dashed] (176,133) -- (176,-9);
%\end{axis}
%\node (x3) at (6,-0.3) {$z_2^*$};
%\node (x4) at (2.8,-0.25) {$z_1^* = $};
%\end{tikzpicture}
%\caption{Nonanalytic $C^{\infty}$ case} \label{fig:S2}
%\end{figure}

%----------------------
\subsection{Equilibrium measures II}
%----------------------

The following statement is a version of the existence result in Theorem~\ref{UD1} for $d>1$.

\begin{theorem}\label{HDC}
Let $(T,\Phi)$ be a $C^r$ regular pair and let $F\colon L(\Phi) \to \R$ be a strictly concave $C^r$ function such that $E = h + F$ attains its maximum on $\interior L(\Phi)$. Then there exists a unique equilibrium measure for $(F,\Phi)$. Moreover, the equilibrium measure is ergodic.
\end{theorem}

\begin{proof}
By Proposition~\ref{HLP}, the map $z \mapsto h(z)$ is upper semicontinuous on~$L(\Phi)$. Since $F$ is $C^r$ on $L(\Phi)$, we conclude that $z \mapsto E(z)$ is upper semicontinuous on $L(\Phi)$. Together with the compactness of $L(\Phi)$, this guarantees the existence of at least one point in $L(\Phi)$ maximizing the function~$E$. On the other hand, by Propositions~\ref{HLP} and~\ref{CRW} and the strict concavity of~$F$, the function $E$ is strictly concave on $L(\Phi)$ and $C^{r-1}$ on $\interior L(\Phi)$. The concavity property of $E$ implies that there is one maximizer in $\interior L(\Phi)$. Since $K(F,\Phi) \subset \interior L(\Phi)$, there exists a unique point $z^* \in \interior L(\Phi)$ such that $E(z^*)$ is a maximal value. It follows from Theorem~\ref{THM} that $K(F,\Phi) = \{z^*\}$, that is, $\nu_{z^*}$ is the unique equilibrium measure for $(F,\Phi)$. Moreover, by Lemma~\ref{BSS}, $\nu_{z^*}$ is an ergodic measure.
\end{proof}

In Example~\ref{exa6}, we have $h|_{\partial L(\phi)} \equiv 0$. It turns out that this behavior at the boundary of $L(\phi)$ is typical for some $C^r$ regular systems, even for $d>1$.
Let $\cH_{\theta}$ be the space of H\"older continuous functions with H\"older exponent $\theta> 0$. The following result is a particular case of Theorem~14 in~\cite{BSS02}.

\begin{theorem}\label{TRS}
Let $T$ be a subshift of finite type, a $C^{1+\epsilon}$ diffeomorphism with a hypebolic set, or a $C^{1+\epsilon}$ map with a repeller, that is assumed to be topologically mixing. Then there exists a residual set $\mathcal{O} \subset (\cH_{\theta})^{d}$ such that for each $\Phi \in \mathcal{O}$ we have
\begin{equation}\label{aaa}
h|_{\partial L(\Phi)} \equiv 0\quad\text{and} \quad
L(\Phi) = \overline{\interior L(\Phi)}.
\end{equation}
\end{theorem}

We also note that in Example~\ref{exa6} with $\alpha > 0$, the function $F$ satisfies
\begin{equation}\label{ttt}
F|_{\interior L(\phi)} > \max_{z \in \partial L(\phi)}F(z),
\end{equation}
where $F_{\interior L(\phi)}$ is the restriction to $\interior L(\phi)$.
This is an example of a more general situation in which $E = h + F$ attains its maximum on $\interior L(\Phi)$:
\begin{equation}\label{aaa2}
\max_{z \in \interior L(\Phi)}E(z) > \max_{z \in \partial L(\Phi)}E(z).
\end{equation}
Note that this condition may depend not only on $F$, but also on the family of functions~$\Phi$.

A similar idea works for typical $C^r$ regular systems in the sense that they belong to the residual set $\mathcal{O}$ in Theorem \ref{TRS}. Let $(T, \Phi)$ be a $C^r$ regular pair satisfying~\eqref{aaa}. In particular, $\interior L(\Phi) \ne \emptyset$. Now let $F\colon \R^d \to \R$ be a function satisfying~\eqref{ttt} with $\phi$ replaced by~$\Phi$. Since $h \ge 0$, we have
\[
\begin{split}
\max_{z \in \partial L(\Phi)}E(z) &\le \max_{z \in \partial L(\Phi)}h(z) + \max_{z \in \partial L(\Phi)}F(z)\\
& \le h|_{\interior L(\Phi)} + \max_{z \in \partial L(\Phi)}F(z) \\
& < h|_{\interior L(\Phi)} + F|_{\interior L(\Phi)} = E_{\interior L(\Phi)},
\end{split}
\]
which implies that property \eqref{aaa2} holds. Therefore, $K(F, \Phi) \subset \interior L(\Phi)$ and so one can apply item (3) of Theorem~\ref{THM}.

For $d=1$ and $C^r$ regular systems, the function $h$ in~\eqref{hh} is strictly concave. The next examples illustrate that this may still happen for $d>1$, but unfortunately we are not able to describe for which $C^r$ regular systems the function $h$ is strictly concave.

\begin{example}[A system with $\interior L(\Phi) = \emptyset$]\label{exa8}
Let $T\colon X \to X$ be the two-sided shift with $X = \{1,2\}^\Z$. Moreover, consider the characteristic functions $\phi_1 = \chi_{C_1}$ and $\phi_2 = \chi_{C_2}$, where $C_i$ is the set of all sequences
\[
(\cdots \omega_{-1}\omega_0\omega_1 \cdots) \in X
\]
with $\omega_0 = i$. For each $\mu \in \cM$, we have $\int_X \phi_i= \mu(C_i)$ for $i=1,2$ and so
\[
L(\Phi) = \bigl\{(\mu(C_1),\mu(C_2)): \mu \in \cM \bigr\}.
\]
By Theorem~8 in~\cite{BSS02}, we have
\[
\begin{split}
h(z_1, z_2) &= \max_{\mu \in \cM}\bigl\{h_\mu(T): (\mu(C_1), \mu(C_2)) = (z_1,z_2)\bigr\}\\
& = -z_1\log z_1 - z_2 \log z_2.
\end{split}
\]
Since $\mu(C_1) + \mu(C_2) = 1$ for each $\mu \in \cM$, we obtain
\[
L(\Phi)=\bigl\{(z_1,z_2) \in [0,1] \times [0,1]: z_1+z_2 = 1\bigr\}
\]
and so $\interior L(\Phi) = \emptyset$. Moreover,
\[
h(z) = -z\log z -(1-z)\log(1-z)\quad\text{for $z \in [0,1]$}.
\]
One can easily verify that $(z_1,z_2) \mapsto h(z_1,z_2)$ is strictly concave on~$L(\Phi)$.
\end{example}

\begin{example}
Let $T\colon X \to X$ be the two-sided shift with $X = \{1,2,3\}^\Z$
and let $\phi_1 = \chi_{C_1}$ and $\phi_2 = \chi_{C_3}$. Proceeding as in Example~\ref{exa8}, we obtain
\[
h(z_1,z_2) = -z_1\log z_1 - z_2 \log z_2 - (1-z_1-z_2)\log(1-z_1-z_2)
\]
and
\[
L(\Phi) = \bigl\{(z_1,z_2) \in [0,1] \times [0,1]: z_1+ z_2 \le 1\bigr\}.
\]
Note that $\interior L(\Phi) \ne \emptyset$ and that $\partial L(\Phi)$ is the set
\[
\bigl((\R\times\{0\}) \cup (\{0\}\times\R) \cup \{(z_1,z_2): z_1+ z_2 = 1\}\bigr) \cap ([0,1] \times [0,1]).
\]
Observe that for $(z_1,z_2) = (1/2,0)$, $(0,1/2)$, $(1/2,1/2) \in \partial L(\Phi)$ we have $h(z_1,z_2) = 1 > 0$ and so the system is not typical. On the other hand, the map $(z_1,z_2) \mapsto h(z_1,z_2)$ is still strictly concave on~$L(\Phi)$.

Now consider the function $F(z_1, z_2) = \beta(z_1^2+z_2^2)/2$ with $\beta \in \R$. One can verify that the determinant of the Hessian matrix of $E = h + F$ is given~by
\[
\det H_{E}(z_1,z_2) = \beta^2 + \beta\frac{z_1(1-z_1) + z_2(1-z_2)}{z_1z_2(1-z_1-z_2)} + \frac1{z_1z_2(1-z_1-z_2)}.
\]
Since $\det H_{E}(z_1,z_2) > 0$ for $(z_1,z_2) \in \interior L(\Phi)$ and $\beta \ge 0$, every critical point of $E$ is nondegenerate for all $\beta \ge 0$. Hence, for each $\beta \ge0$, the function $E$ has at most finitely many critical points. In addition, if condition \eqref{aaa2} holds, then $E$ attains its maximal value only at critical points. Hence, by Theorem~\ref{THM} the pair $(F,\Phi)$ has finitely many equilibrium measures. For a specific example, for $\beta = 1$ one can verify that condition \eqref{aaa2} holds and so $(F,\Phi)$ has finitely many equilibrium measures.

On the other hand, for $\beta < 0$ the function $F$ is strictly concave and one can use Theorem~\ref{HDC} to conclude that $(F,\Phi)$ has a unique equilibrium measure.
\end{example}

%----------------------
\subsection{Coincidence of equilibrium measures}
%----------------------

The following result gives a sufficient condition so that two systems share equilibrium measures.
We say that $\Phi_1 = \{\phi_{1,1},\ldots,\phi_{1,d}\}$ is \emph{cohomologous} to $\Phi_2 = \{\phi_{2,1},\ldots,\phi_{2,d}\}$ if $\phi_{1,i}$ is cohomologous to $\phi_{2,i}$ for $i =1,\ldots,d$. Then
\[
\int_X\Phi_1 \,d\mu = \int_X\Phi_2 \,d\mu\quad\text{for each $\mu \in \cM$,}
\]
which readily implies that $L(\Phi_1) = L(\Phi_2)$.

\begin{proposition}\label{F1F2}
Let $(T, \Phi_1)$ and $(T,\Phi_2)$ be $C^r$ regular pairs such that $\Phi_1$ is cohomologous to $\Phi_2$ and let $F_1\colon L(\Phi_1) \to \R$ and $F_2\colon L(\Phi_2) \to \R$ be $C^r$ functions. If $z \in \interior L(\Phi_1) \cap \interior L(\Phi_2)$ is simultaneously a maximizer of the functions $E_1= h_1 + F_1$ and $E_2= h_2 + F_2$, then $\nu_{z}$ is an equilibrium measure for $(F_1, \Phi_1)$ and $(F_2,\Phi_2)$.
\end{proposition}

\begin{proof}
Since $\Phi_1$ is cohomologous to $\Phi_2$, we have $L:= L(\Phi_1) = L(\Phi_2)$. Now take $z \in \interior L$ and consider the functions
\[
L_1(q)= P(\langle q, \Phi_1 - z\rangle) - h_1(z)\quad \text{and}\quad
L_2(q)= P(\langle q, \Phi_2-z\rangle) - h_2(z),
\]
where $P$ denotes the classical topological pressure and where each $h_i$ is the corresponding entropy function (see~\eqref{hh}). By the cohomology assumption, we have
\[
\lim_{n \to \infty}\frac{\|S_n\phi_{1,i} - S_n\phi_{2,i}\|_{\infty}}{n} = 0 \quad \text{for $i =1,\ldots,d$}
\]
and so $C_{z}(\Phi_1) = C_{z}(\Phi_2)$ for all $z \in \R^d$ (see~\eqref{Cz}). In particular, this implies that $h := h_1 = h_2$. Therefore,
\[
[\langle q, \Phi_1-z\rangle - h_1(z)]-[\langle q, \Phi_2-z\rangle - h_2(z)] = \langle q, \Phi_1 - \Phi_2\rangle
\]
for $q \in \R^d$. Again since $\Phi_1$ is cohomologous to $\Phi_2$, we conclude that
\begin{equation}\label{gqq}
L_1(q) = L_2(q)\quad\text{for $q \in \R^d$.}
\end{equation}

On the other hand, by the proof of Theorem~8 in~\cite{BSS02} the function $q\mapsto L_1(q)$ attains its minimum at a point $q_1(z)$ and $\nu_{1,z}$ is the unique equilibrium measure for the function $\langle q_1(z), \Phi_1 - z\rangle - h(z)$. Similarly, $q \mapsto L_2(q)$ attains its minimum at a point $q_2(z)$ and $\nu_{2,z}$ is the unique equilibrium measure for the function $\langle q_2(z), \Phi_2 - z\rangle - h(z)$. By \eqref{gqq}, one can take $q_1(z) = q_2(z)$ and so $\nu_z:= \nu_{1,z} = \nu_{2,z}$. The desired result follows now from Theorem~\ref{THM}.
\end{proof}

A direct consequence of Proposition~\ref{F1F2} is that if $\Phi_1$ is cohomologous to~$\Phi_2$ and the functions $E_1$ and $E_2$ attain maximal values at the same points, then $(F_1, \Phi_1)$ and $(F_2, \Phi_2)$ have the same equilibrium measures (in~particular, this happens when $F_1 = F_2$).
For the converse to hold we need stronger conditions so that the coincidence of two equilibrium measures yields a cohomology relation.

\begin{theorem}
Let $X$ be a topologically mixing locally maximal hyperbolic set for a diffeomorphism~$T$ and let $\Phi_1$ and $\Phi_2$ be families of H\"older continuous functions.
Moreover, let $F_1$ and $F_2$ be functions such that $E_1$ and $E_2$ attain maximal values in $\interior L(\Phi_1)$ and $\interior L(\Phi_2)$, respectively. If~$(F_1,\Phi_1)$ and $(F_2,\Phi_2)$ have the same equilibrium measures, then for each $z_1$ and~$z_2$ maximizing $E_1$ and~$E_2$, respectively, there exist $q_1, q_2\in \R^d$ such that $\langle q_1, \Phi_1 - z_1\rangle$ is cohomologous to $\langle q_2, \Phi_2-z_2\rangle$.
\end{theorem}

\begin{proof}
By Theorem~\ref{THM}, each equilibrium measure for $(F_i,\Phi_i)$ is a measure $\nu_{z_i}$ with $z_i \in K(F_i,\Phi_i)$ that is the unique equilibrium measure for
\[
\psi_i = \langle q_i(z_i), \Phi_i - z_i\rangle - h_i(z_i),
\]
where $q_i(z_i)$ is a minimizer of the function
\[
L(q) = P(\langle q, \Phi_i - z_i\rangle) - h_i(z_i).
\]
Since by hypotheses $\nu_{z_1} = \nu_{z_2}$, the function $\psi_1-\psi_2$ is cohomologous to $P(\psi_1) - P(\psi_2) \in \R$. But since
\[
L_1(q_1(z_1)) = L_2(q_2(z_2)) = 0
\]
(see Lemma~2 in~\cite{BSS02}), we have $P(\psi_1) =P(\psi_2)$. So there exists a continuous function $S= S(z_1,z_2): X \to \R$ such that $\psi_1- \psi_2 = S\circ T - S$, that is,
\[
\begin{split}
S \circ T - S & = \langle q_1(z_1), \Phi_1 - z_1\rangle - \langle q_2(z_2), \Phi_2 - z_2\rangle - h_1(z_1) + h_2(z_2).
\end{split}
\]
Again since $\nu_{z_1} = \nu_{z_2}$, by Lemma~\ref{FZ4} we have
\[
h_1(z_1) = h_{\nu_{z_1}}(T) = h_{\nu_{z_2}}(T) = h_2(z_2).
\]
Hence, for each $z_1$ and~$z_2$ maximizing $E_1$ and~$E_2$, respectively, there exist points $q_1=q_1(z_1), q_2=q_2(z_2) \in \R^d$ as in the statement of the theorem.
\end{proof}

\bibliographystyle{alpha}

\end{document}